\documentclass[12pt,a4paper]{article}
\usepackage{graphicx}
\usepackage[margin=1in]{geometry}
\usepackage{amsmath,amsthm,amssymb,bbm,color}
\usepackage{cite}
\usepackage{hyperref}
\hypersetup{colorlinks=true}

\newtheorem{thm}{Theorem}[section]
\newtheorem{lemma}[thm]{Lemma}
\newtheorem{prop}[thm]{Proposition}

\renewcommand{\P}{\mathbf P}
\newcommand{\E}{\mathbf E}
\newcommand{\wt}{\widetilde}
\newcommand{\F}{\mathcal F}

\numberwithin{equation}{section}

\title{Moments of the superdiffusive elephant random walk with general step distribution}

\author{J\'ozsef Kiss\thanks{Department of Stochastics, Institute of Mathematics,
Budapest University of Technology and Economics, M\H uegyetem rkp.\ 3., H-1111 Budapest, Hungary. E-mail: {\tt kissj@math.bme.hu}}
\and B\'alint Vet\H o\thanks{Department of Stochastics, Institute of Mathematics,
Budapest University of Technology and Economics, M\H uegyetem rkp.\ 3., H-1111 Budapest, Hungary. E-mail: {\tt vetob@math.bme.hu}}
\thanks{ELKH--BME Stochastics Research Group, M\H uegyetem rkp.\ 3., H-1111 Budapest, Hungary}}

\begin{document}

\maketitle

\begin{abstract}
We consider the elephant random walk with general step distribution.
We calculate the first four moments of the limiting distribution of the position rescaled by $n^\alpha$ in the superdiffusive regime
where $\alpha$ is the memory parameter.
This extends the results obtained by Bercu in~\cite{Bercu17}.
\end{abstract}

\section{Introduction and results}
The elephant random walk (ERW) is a one-dimensional discrete-time random walk with memory.
With probability $\alpha$ the walker repeats one of its previous steps chosen uniformly at random
and with probability $1-\alpha$ the next step is sampled independently from the past where $\alpha\in[0,1]$ is the memory parameter.

The ERW was first introduced in 1993 by Drezner and Farnum~\cite{DF93} as correlated Bernoulli process
with Bernoulli step distribution and time-dependent memory parameter.
For the case of time-homogeneous memory parameter and Bernoulli step distribution it was proved in~\cite{H04}
that the behaviour shows a phase transition in the value of the memory parameter $\alpha$.
In the diffusive regime ($\alpha<1/2$) asymptotic normality is proved after diffusive scaling,
in the critical regime ($\alpha=1/2$) normality remains valid with a logarithmic correction in the scaling.
In the superdiffusive regime ($\alpha>1/2$), after scaling with $n^{\alpha}$, the limiting distribution is found to be non-degenerate.
It was also stated without proof that the limiting distribution is different from the normal distribution.
The proof uses the martingale which naturally appears in the problem.
In the case of general time-dependent memory parameter sufficient conditions for the law of large numbers,
central limit theorem and law of iterated logarithm were given in~\cite{JJQ08} using the martingale method.

The same model with $+1$ and $-1$ jumps was first named as elephant random walk in~\cite{ST04}
and the probability distribution of its position after $n$ steps was analysed.
The connection of the ERW with P\'olya-type urns was exploited in~\cite{BB16} to prove process convergence of the ERW trajectory using know results on urns.
The fact that the limiting distribution of the superdiffusive ERW is not Gaussian was first proved rigorously in~\cite{Bercu17}
by computing its first four moments using martingales.
New hypergeometric identities are obtained in~\cite{BCR19} by computing these moments in two different ways.
Number of zeros in the elephant random walk is analysed in~\cite{Bertoin22}.
The generalization when zero jumps are also allowed is called the delayed ERW, see~\cite{GS21,Bercu22}.
In~\cite{Bus18} the steps of the ERW are sampled from the $\beta$-stable distribution with parameter $\beta\in(0,2]$
and the phase transition in the memory parameter is proved to happen at the value $\alpha=1/\beta$ using the connection with random recursive trees.
In the superdiffusive regime the fluctuations after subtracting the non-Gaussian limit are proved to be normal in~\cite{KT19}.

In the present note we consider the ERW with general step distribution which is defined as follows.
Let $\alpha\in[0,1]$ be the memory parameter of the ERW.
Let $\xi_1,\xi_2,\dots$ be an arbitrary i.i.d.\ sequence of random variables with certain moment conditions imposed later.
We denote by $X_n$ the $n$th step of the random walk.
We suppose that the random walk starts from the origin, i.e., $S_0=0$.
The first step is $X_1=\xi_1$.
Every further step is defined as
\begin{equation}\label{defXn}
X_{n+1}=\begin{cases}X_K&\text{with probability }\alpha,\\\xi_{n+1}&\text{with probability }1-\alpha\end{cases}
\end{equation}
where the index $K$ has uniform distribution on the index set $\{1,2,\dots,n\}$,
that is, with probability $\alpha$ one of the previous steps is repeated and otherwise the step is an independent new sample from the step distribution.
Note that the steps $X_1,X_2,\dots$ are not independent but the walk has a long memory.
The position of the ERW is denoted by
\begin{equation}\label{defSn}
S_{n+1}=S_n+X_{n+1}.
\end{equation}
Let
\begin{equation}\label{defmkMk}
m_k=\E(\xi_1^k),\qquad M_k=\E\left((\xi_1-m_1)^k\right)
\end{equation}
for $k=1,2,\dots$ denote the moments and centered moments of the step distribution.

For the ERW with general step distribution the same phase transition appears in the value of the memory parameter $\alpha$ as for the original model
since the martingale method used in the majority of the previous literature on ERW extends naturally for our model as we explain below.
We believe that the proof of the law of large numbers and central limit theorem in the diffusive and critical regime
survives for the case of general step distribution after appropriate modifications with the same Gaussian limits.
Hence we focus on the most exciting superdiffusive regime in the present note where the limiting distribution is different from Gaussian.
Our main results for the superdiffusive ERW with general step distribution are the following.

\begin{thm}\label{thm:conv}
\begin{enumerate}
\item
Let $(S_n)$ denote the elephant random walk with memory parameter $\alpha$.
Assume that $\alpha\in(1/2,1]$, that is, we consider the superdiffusive regime.
Suppose that the step distribution has finite variance, that is, $m_2<\infty$.
Then
\begin{equation}\label{Qas}
\lim_{n\to\infty}\frac{S_n-nm_1}{n^\alpha}=Q
\end{equation}
almost surely with some non-degenerate random variable $Q$.

\item
Let $p$ be a positive even integer.
Assume that the $p$th absolute moment of the step distribution is finite, that is, $m_p<\infty$.
Then the above convergence is also true in $L^p$ which means that
\begin{equation}\label{QLp}
\lim_{n\to\infty}\E\left(\left|\frac{S_n-nm_1}{n^\alpha}-Q\right|^p\right)=0.
\end{equation}
\end{enumerate}
\end{thm}

\begin{thm}\label{thm:moments}
Assume that the step distribution of the elephant random walk $(S_n)$ has finite fourth moment, that is, $m_4<\infty$.
Then the first four moments of the random variable $Q$ which arises as the limits in \eqref{Qas}--\eqref{QLp} are given by
\begin{align}
\E(Q)&=0,\label{Q1}\\
\E\left(Q^2\right)&=\frac{M_2}{(2\alpha-1)\Gamma(2\alpha)},\label{Q2}\\
\E\left(Q^3\right)&=\frac{4M_3}{(3\alpha-1)\Gamma(3\alpha)},\label{Q3}\\
\E\left(Q^4\right)&=\frac{6(3(2\alpha-1)^2M_4+2(1-\alpha)(5\alpha-2)M_2^2)}{(2\alpha-1)^2(4\alpha-1)\Gamma(4\alpha)}.\label{Q4}
\end{align}
\end{thm}

Theorem~\ref{thm:conv} follows from the application of the martingale method to the case of general step distribution.
The almost sure convergence in \eqref{Qas} was already proved in Theorem 1.1 of~\cite{Bertenghi21}.
The $L^p$ convergence was established for $p=2$ in Theorem 3.2 of~\cite{Bertoin21b}
and for general $p$ in Theorem 2.2 of~\cite{BCR19} for the standard elephant random walk.
See also \cite{Bertoin21a} for other generalizations of these convergence results.
We provide a simple and elementry proof of the almost sure and $L^p$ convergence results of Theorem~\ref{thm:conv} in Section~\ref{s:conv}
which relies on proving the $L^p$ boundedness of the natural martingale if the step distribution has a finite $p$th absolute moment.
In particular we prove the $L^p$ boundedness of a sequence of martingale differences in Lemma~\ref{lemma:Lpbound}.

Theorem~\ref{thm:moments} is proved in Section~\ref{s:moments} by solving the recursions for the mixed moments of the centered ERW.
The moments in \eqref{Q1}--\eqref{Q4} generalize the formulas found in~\cite{Bercu17} in the case of symmetric first step.
We mention that higher moments of $Q$ could in principle be determined using the method presented here but the recursions are much more complicated
beyond the fourth moment.

\paragraph{Acknowledgments.}
The work of the authors was supported by the NKFI (National Research, Development and Innovation Office) grant FK123962.
B.\ Vet\H o was supported by the Bolyai Research Scholarship of the Hungarian Academy of Sciences
and by the \'UNKP--21--5 New National Excellence Program of the Ministry for Innovation and Technology
from the source of the National Research, Development and Innovation Fund.

\section{Martingale method and convergence}
\label{s:conv}

We assume that the first two moments of the step distribution are finite.
Let
\begin{equation}\label{tS_n}
\wt S_n=S_n-nm_1=\sum_{k=1}^n X_k-nm_1
\end{equation}
denote the centered ERW.
Then by the definition \eqref{defXn} we have for any $n=1,2,\dots$ that
\begin{equation}\label{Xcondexp}
\E(X_{n+1}-m_1|\mathcal{F}_n)=\frac{\alpha}{n}\wt S_n
\end{equation}
where $\F_n=\sigma(X_1,\dots,X_n)$ is the natural filtration.
As a consequence,
\begin{equation}
\E(\wt S_{n+1}|\mathcal{F}_n)=\left(1+\frac\alpha n\right)\wt S_n
\end{equation}
holds and the process
\begin{equation}
Q_n=a_n\wt S_n
\end{equation}
is a martingale with respect to $\F_n$ where the sequence $(a_n)$ is given by
\begin{equation} \label{aenn}
a_n=\Gamma(1+\alpha)^{-1}\prod_{k=1}^{n-1}\left(1+\frac\alpha k\right)^{-1}=\frac{\Gamma(n)}{\Gamma(n+\alpha)}\sim n^{-\alpha}
\end{equation}
as $n\to\infty$ with the empty product understood to be equal to $1$ in the definition of $a_1=\Gamma(1+\alpha)^{-1}$.
We mention that our definition \eqref{aenn} of $a_n$ compared to the literature is simplified by a factor $\Gamma(1+\alpha)$, see e.g.~\cite{Bercu17}.

The martingale $(Q_n)$ can be written as
\begin{equation}\label{Qnsumrepr}
Q_n=\sum_{k=1}^na_k\varepsilon_k
\end{equation}
where $\varepsilon_1=X_1-m_1$ and for all $k=2,3,\dots$,
\begin{equation}\label{defeps}
\varepsilon_k=\wt S_k-\left(1+\frac\alpha{k-1}\right)\wt S_{k-1}=X_k-\E(X_k|\mathcal F_{k-1}).
\end{equation}

\begin{lemma}\label{lemma:Lpbound}
Let $p$ be a positive integer and assume that the $p$th absolute moment of the step distribution is finite.
Then the martingale differences $(\varepsilon_n)$ are bounded in $L^p$ and
\begin{equation}\label{Lpbound}
\sup_{n\ge1}\E(|\varepsilon_n|^p)\le2^p\E(|\xi_1|^p).
\end{equation}
\end{lemma}

\begin{proof}[Proof of Lemma~\ref{lemma:Lpbound}]
We first use induction to see that $\E(|X_n|^p)=\E(|\xi_1|^p)$.
The statement is clear for $n=1$ and for $n=2,3,\dots$, one can write by the law of total expectation that
\begin{equation}
\E(|X_n|^p)=\sum_{k=1}^{n-1}\E\left(|X_n|^p \big| X_n=X_k\right)\P(X_n=X_k)+\E\left(|X_n|^p\big|X_n=\xi_n\right)\P(X_n=\xi_n)
\end{equation}
which is equal to $\E(|\xi_1|^p)$ by the induction hypothesis.

On the other hand, Jensen's inequality implies that $|\E(X_n|\F_{n-1})|^p\le\E(|X_n|^p|\F_{n-1})$,
which after taking expectation yields that $\E(|\E(X_n|\F_{n-1})|^p)\le\E(|\xi_1|^p)$.
Then by applying the Minkowski inequality for $\varepsilon_n=X_n-\E(X_n|\F_{n-1})$ from \eqref{defeps}, we have that
\begin{equation}
\E(|\varepsilon_n|^p)\le\left(\left(\E\left(|X_n|^p\right)\right)^{1/p}+\left(\E\big(|\E\left(X_n|\F_{n-1}\right)|^p\big)\right)^{1/p}\right)^p
\le2^p\E(|\xi_1|^p)
\end{equation}
which proves \eqref{Lpbound}.
\end{proof}

\begin{proof}[Proof of Theorem~\ref{thm:conv}]
\begin{enumerate}
\item
It is clear from the representation \eqref{Qnsumrepr} and from Lemma~\ref{lemma:Lpbound}
that the expectation of the predictable quadratic variation process can be bounded as
\begin{equation}
\E(\langle Q\rangle_n)\le4m_2\sum_{k=1}^na_k^2
\end{equation}
which remains finite in $n$ exactly in the superdiffusive regime $\alpha\in(1/2,1]$.
As a consequence, the increasing limit $\lim_{n\to\infty}\langle Q\rangle_n$ is an almost surely finite random variable
and the martingale $(Q_n)$ converges almost surely to its limit $Q=\sum_{k=1}^\infty a_k\varepsilon_k$.

\item
The conditional expectation of the $p$th power of $Q_{n+1}$ using $Q_{n+1}=Q_n+a_{n+1}\varepsilon_{n+1}$ from \eqref{Qnsumrepr} can be written as
\begin{equation}\label{Qn+1condexp}
\E(Q_{n+1}^p|\F_n)=\sum_{k=0}^p\binom pka_{n+1}^k\E(\varepsilon_{n+1}^k|\F_n)Q_n^{p-k}.
\end{equation}
Note that the $k=1$ term above vanishes since $\E(\varepsilon_{n+1}|\F_n)=0$.
The absolute value of the expectation of the random variable which appears in the $k=2,\dots,p$ terms on the right-hand side of \eqref{Qn+1condexp}
can be upper bounded as
\begin{equation}\label{QnHolder}\begin{aligned}
\left|\E\left(\E(\varepsilon_{n+1}^k|\F_n)Q_n^{p-k}\right)\right|&\le\E\left(\E(|\varepsilon_{n+1}|^k|\F_n)|Q_n|^{p-k}\right)\\
&\le\left(\E\left(\big(\E(|\varepsilon_{n+1}|^k|\F_n)\big)^{p/k}\right)\right)^{k/p}\big(\E\left(|Q_n|^p\right)\big)^{(p-k)/p}\\
&\le\left(\E\big(\E(|\varepsilon_{n+1}|^p|\F_n)\big)\right)^{k/p}\big(\E\left(|Q_n|^p\right)\big)^{(p-k)/p}\\
&=\big(\E(|\varepsilon_{n+1}|^p)\big)^{k/p}\big(\E\left(|Q_n|^p\right)\big)^{(p-k)/p}
\end{aligned}\end{equation}
where we used H\"older's inequality in the second inequality above and Jensen's inequality for conditional expectations in the last one.
By taking expectation in \eqref{Qn+1condexp} we get that
\begin{equation}\label{EQnp}\begin{aligned}
\E(Q_{n+1}^p)&\le\E(Q_n^p)+a_{n+1}^2\sum_{k=2}^p\binom pk\big(\E(|\varepsilon_{n+1}|^p)\big)^{k/p}\big(\E(Q_n^p)\big)^{(p-k)/p}\\
&\le\E(Q_n^p)+a_{n+1}^2\sum_{k=2}^p\binom pk(1+\E(|\varepsilon_{n+1}|^p))(1+\E(Q_n^p))\\
&\le\left(1+a_{n+1}^22^p(1+\E(|\varepsilon_{n+1}|^p))\right)\E(Q_n^p)+a_{n+1}^22^p(1+\E(|\varepsilon_{n+1}|^p))
\end{aligned}\end{equation}
where we used \eqref{QnHolder} and the fact that $a_{n+1}\in(0,1]$ in the first inequality above
and the upper bounds $(\E(|\varepsilon_{n+1}|^p))^{k/p}\le1+\E(|\varepsilon_{n+1}|^p)$ and $(\E(Q_n^p))^{(p-k)/p}\le1+\E(Q_n^p)$ in the second one.
Note also that since $p$ is even, we have $\E(Q_n^p)=\E(|Q_n|^p)$.
By Lemma~\ref{lemma:Lpbound}, we have $\E(|\varepsilon_{n+1}|^p)\le2^p\E(|\xi_1|^p)$ where the upper bound does not depend on $n$.
By Lemma~\ref{lemma:recursion2} below with $\beta=2\alpha$ and $c=2^p(1+2^p\E(|\xi_1|^p))$, the expectations $\E(Q_n^p)$ remain bounded in $n$,
that is, the martingale $(Q_n)$ is bounded in $L^p$, hence it converges to its limit $Q$ also in $L^p$.
\end{enumerate}
\end{proof}

\begin{lemma}\label{lemma:recursion2}
Let $(b_n)$ be a sequence of positive real numbers which satisfies the recursive inequality
\begin{equation}
b_{n+1}\le\left(1+\frac c{n^{\beta}}\right)b_n+\frac c{n^{\beta}}
\end{equation}
for some $\beta>1$ and $c>0$ and suppose that $b_1\le c$.
Then the sequence $(b_n)$ remains bounded in $n$.
\end{lemma}

\begin{proof}[Proof of Lemma~\ref{lemma:recursion2}]
It can be seen by induction on $n$ that
\begin{equation}\label{bnupperbound}
b_n\le\prod_{k=1}^{n-1}\left(1+\frac c{k^\beta}\right)\left(c+\sum_{k=1}^{n-1}\frac c{k^\beta}\right)
\end{equation}
holds for all $n=1,2,\dots$.
The upper bound on the right-hand side of \eqref{bnupperbound} is increasing in $n$ and its $n\to\infty$ limit is finite since $\beta>1$.
\end{proof}

\section{Limiting moments}
\label{s:moments}

We give the proof of Theorem~\ref{thm:moments} in this section.
For this we introduce
\begin{equation}
\wt T_n=\sum_{k=1}^n X_k^2-nm_2,\qquad\wt U_n=\sum_{k=1}^n X_k^3-nm_3.
\end{equation}
We define the mixed moments
\begin{align}
M_{1,2}&=\E((\xi_1-m_1)(\xi_1^2-m_2))=m_3-m_1m_2,\label{defM12}\\
M_{1,3}&=\E((\xi_1-m_1)(\xi_1^3-m_3))=m_4-m_1m_3,\label{defM13}\\
M_{2,2}&=\E((\xi_1^2-m_2)^2)=m_4-m_2^2,\label{defM22}\\
M_{1,1,2}&=\E((\xi_1-m_1)^2(\xi_1^2-m_2))=m_4-m_2^2-2m_1m_3+2m_1^2m_2.\label{defM112}
\end{align}
Note that the moments $M_k$ given in \eqref{defmkMk} can be expressed in terms of the moments $m_k$ as
\begin{align}
M_2&=m_2-m_1^2,\label{defM2}\\
M_3&=m_3-3m_1m_2+2m_1^3,\label{defM3}\\
M_4&=m_4-4m_1m_3+6m_1^2m_2-3m_1^4.\label{defM4}
\end{align}
The idea to compute the moments of the limit $Q$ in Theorem~\ref{thm:moments} is to use the convergence in $L^p$ from Theorem~\ref{thm:conv} with $p=4$
and to write down and solve recursions for the mixed moments of the elephant random walk,
see Propositions~\ref{prop:recursions} and \ref{prop:recsolutions} below.

\begin{prop}\label{prop:recursions}
The mixed moments of $\wt S_n$, $\wt T_n$ and $\wt U_n$ satisfy the following recursions:
\begin{align}
\E(\wt S_{n+1}^2)&=\left(1+\frac{2\alpha}n\right)\E(\wt S_n^2)+M_2,\label{recursion2}\\
\E(\wt S_{n+1}\wt T_{n+1})&=\left(1+\frac{2\alpha}n\right)\E(\wt S_n\wt T_n)+M_{1,2},\label{recursion12}\\
\E(\wt S_{n+1}^3)&=\left(1+\frac{3\alpha}n\right)\E(\wt S_n^3)+\frac{3\alpha}n\E(\wt S_n\wt T_n)-\frac{6\alpha}nm_1\E(\wt S_n^2)+M_3,\label{recursion3}\\
\E(\wt S_{n+1}\wt U_{n+1})&=\left(1+\frac{2\alpha}n\right)\E(\wt S_n\wt U_n)+M_{1,3},\label{recursion13}\\
\E(\wt T_{n+1}^2)&=\left(1+\frac{2\alpha}n\right)\E(\wt T_n^2)+M_{2,2},\label{recursion22}\\
\E(\wt S_{n+1}^2\wt T_{n+1})
&=\left(1+\frac{3\alpha}n\right)\E(\wt S_n^2\wt T_n)+\frac{2\alpha}n\E(\wt S_n\wt U_n)+\frac\alpha n\E(\wt T_n^2)\label{recursion112}\\
&\qquad-\frac{4\alpha}nm_1\E(\wt S_n\wt T_n)-\frac{2\alpha}nm_2\E(\wt S_n^2)+M_{1,1,2},\notag\\
\E(\wt S_{n+1}^4)&=\left(1+\frac{4\alpha}n\right)\E(\wt S_n^4)+\frac{6\alpha}n\E(\wt S_n^2\wt T_n)+\frac{4\alpha}n\E(\wt S_n\wt U_n)\label{recursion4}\\
&\qquad-\frac{12\alpha}nm_1\left(\E(\wt S_n^3)+\E(\wt S_n\wt T_n)\right)+\left(\frac{12\alpha}nm_1^2+6M_2\right)\E(\wt S_n^2)+M_4\notag
\end{align}
for $n=1,2,\dots$ where the initial values are given by
\begin{equation}\begin{aligned}
\E(\wt S_1^2)&=M_2,&\E(\wt S_1\wt T_1)&=M_{1,2},&\E(\wt S_1^3)&=M_3,&\E(\wt S_1\wt U_1)&=M_{1,3},\\
\E(\wt T_1^2)&=M_{2,2},&\E(\wt S_1^2\wt T_1)&=M_{1,1,2},&\E(\wt S_1^4)&=M_4.
\end{aligned}\end{equation}
\end{prop}

\begin{prop}\label{prop:recsolutions}
The solutions of the recursions written in Proposition~\ref{prop:recursions} are given by
\begin{align}
\E(\wt S_n^2)&=\frac{M_2}{(2\alpha-1)\Gamma(2\alpha)}\frac{\Gamma(n+2\alpha)}{\Gamma(n)}-\frac{M_2}{2\alpha-1}n,\label{recsolution2}\\
\E(\wt S_n\wt T_n)&=\frac{M_{1,2}}{(2\alpha-1)\Gamma(2\alpha)}\frac{\Gamma(n+2\alpha)}{\Gamma(n)}-\frac{M_{1,2}}{2\alpha-1}n,\label{recsolution12}\\
\E(\wt S_n^3)&=\frac{4M_3}{(3\alpha-1)\Gamma(3\alpha)}\frac{\Gamma(n+3\alpha)}{\Gamma(n)}-\frac{3M_3}{(2\alpha-1)\Gamma(2\alpha)}\frac{\Gamma(n+2\alpha)}{\Gamma(n)}
+\frac{(\alpha+1)M_3}{(2\alpha-1)(3\alpha-1)}n,\label{recsolution3}\\
\E(\wt S_n\wt U_n)&=\frac{M_{1,3}}{(2\alpha-1)\Gamma(2\alpha)}\frac{\Gamma(n+2\alpha)}{\Gamma(n)}-\frac{M_{1,3}}{2\alpha-1}n,\label{recsolution13}\\
\E(\wt T_n^2)&=\frac{M_{2,2}}{(2\alpha-1)\Gamma(2\alpha)}\frac{\Gamma(n+2\alpha)}{\Gamma(n)}-\frac{M_{2,2}}{2\alpha-1}n,\label{recsolution22}\\
\E(\wt S_n^2\wt T_n)&=\frac{4M_{1,1,2}}{(3\alpha-1)\Gamma(3\alpha)}\frac{\Gamma(n+3\alpha)}{\Gamma(n)}-\frac{3M_{1,1,2}}{(2\alpha-1)\Gamma(2\alpha)}\frac{\Gamma(n+2\alpha)}{\Gamma(n)}
+\frac{(\alpha+1)M_{1,1,2}}{(2\alpha-1)(3\alpha-1)}n.\label{recsolution112}
\end{align}
Finally the asymptotic equivalence
\begin{equation}\label{recsolution4}
\E(\wt S_n^4)\sim\frac{6(3(2\alpha-1)^2M_4+2(1-\alpha)(5\alpha-2)M_2^2)}{(2\alpha-1)^2(4\alpha-1)\Gamma(4\alpha)}\frac{\Gamma(n+4\alpha)}{\Gamma(n)}
\end{equation}
holds as $n\to\infty$.
\end{prop}

\begin{proof}[Proof of Theorem~\ref{thm:moments}]
By the assumption $m_4<\infty$, Theorem~\ref{thm:conv} implies the convergence \eqref{QLp} in $L^p$ for $p=4$.
As a consequence
\begin{equation}\label{momentconv}
n^{-p\alpha}\E(\wt S_n^p)\to\E(Q^p)
\end{equation}
as $n\to\infty$ for $p=1,\dots,4$.
Note that $\E(X_n)=m_1$ for all $n$ hence $\E(\wt S_n)=0$ which implies \eqref{Q1}.
The values of the moments in \eqref{Q2}--\eqref{Q4} are obtained by taking the leading term
in formulas \eqref{recsolution2}, \eqref{recsolution3} and \eqref{recsolution4} using \eqref{momentconv}
and the asymptotic equality $\Gamma(n+p\alpha)/\Gamma(n)\sim n^{p\alpha}$.
\end{proof}

\begin{proof}[Proof of Proposition~\ref{prop:recursions}]
We start by writing
\begin{equation}\label{recursionrewrite}
\wt S_{n+1}=\wt S_n+X_{n+1}-m_1,\quad\wt T_{n+1}=\wt T_n+X_{n+1}^2-m_2,\quad\wt U_{n+1}=\wt U_n+X_{n+1}^3-m_3.
\end{equation}
We use these formulas on the left-hand side of the recursions \eqref{recursion2}--\eqref{recursion4} and we expand the products under the expectation.
Then we get the sum of several expectations involving products with combinations of $\wt S_n$, $\wt T_n$, $\wt U_n$ multiplied by powers of $X_{n+1}$.
The expectation of such a product is computed by taking the conditional expectation of the factor involving $X_{n+1}$ with respect to $\F_n$ first
and then by taking expectation, e.g.
\begin{equation}\label{examplerec}
\E(\wt S_{n+1}^4)=\sum_{k=0}^4\binom4k\E(\wt S_n^k(X_{n+1}-m_1)^{4-k})=\sum_{k=0}^4\binom4k\E(\wt S_n^k\E((X_{n+1}-m_1)^{4-k}|\F_n)).
\end{equation}
There are two types of terms in the resulting expressions:
mixed terms including powers of $X_{n+1}$ multiplied by an expression of $\wt S_n$, $\wt T_n$ or $\wt U_n$ under the expectation
($k=1,2,3,4$ terms in \eqref{examplerec})
and pure terms being the expectation of a polynomial of $X_{n+1}$ only ($k=0$ term in \eqref{examplerec}).
In order to compute the expectation appearing in mixed terms, we use the conditional expectation of powers of $X_{n+1}$ given in \eqref{condexp1}--\eqref{condexp3} below.
For the pure terms, the computation of the conditional expectation of the appropriate polynomial is not needed,
the expectations given in \eqref{expectations:mixed} are enough to get the recursions \eqref{recursion2}--\eqref{recursion4} for the expectations.

Next we prove
\begin{align}
\E(X_{n+1}-m_1|\F_n)&=\frac\alpha n\wt S_n,\label{condexp1}\\
\E((X_{n+1}-m_1)^2|\F_n)&=\frac\alpha n\wt T_n-\frac{2\alpha}nm_1\wt S_n+M_2,\label{condexp11}\\
\E((X_{n+1}-m_1)^3|\F_n)&=\frac\alpha n\wt U_n-\frac{3\alpha}nm_1\wt T_n+\frac{3\alpha}nm_1^2\wt S_n+M_3,\label{condexp111}\\
\E(X_{n+1}^2-m_2|\F_n)&=\frac\alpha n\wt T_n,\label{condexp2}\\
\E((X_{n+1}^2-m_2)(X_{n+1}-m_1)|\F_n)&=\frac\alpha n\wt U_n-\frac\alpha nm_1\wt T_n-\frac\alpha nm_2\wt S_n+M_{1,2},\label{condexp12}\\
\E(X_{n+1}^3-m_3|\F_n)&=\frac\alpha n\wt U_n.\label{condexp3}
\end{align}
To see \eqref{condexp1}, \eqref{condexp2} and \eqref{condexp3} it is enough to observe that for any integer $k$,
$X_{n+1}^k$ is equal to $X_i^k$ for some $i=1,\dots,n$ with probability $\alpha/n$ each
and it is equal to a new sample $\xi_{n+1}^k$ with probability $1-\alpha$.
Then the conditional expectation is equal to $\frac\alpha n\sum_{i=1}^nX_i^k-\alpha m_k$.

For \eqref{condexp11}, we write
\begin{equation}
(X_{n+1}-m_1)^2=(X_{n+1}^2-m_2)-2m_1(X_{n+1}-m_1)+m_2-m_1^2
\end{equation}
and we apply \eqref{condexp2} and \eqref{condexp1}.
We get \eqref{condexp111} by writing
\begin{equation}
(X_{n+1}-m_1)^3=(X_{n+1}^3-m_3)-3m_1(X_{n+1}^2-m_2)+3m_1^2(X_{n+1}-m_1)+M_3.
\end{equation}
We use
\begin{equation}
(X_{n+1}^2-m_2)(X_{n+1}-m_1)=(X_{n+1}^3-m_3)-m_1(X_{n+1}^2-m_2)-m_2(X_{n+1}-m_1)+M_{1,2}
\end{equation}
for the proof of \eqref{condexp12}.

Further using the definitions \eqref{defM4}, \eqref{defM22}, \eqref{defM13} and \eqref{defM112}, we can see by induction on $n$ the equality of expectations
\begin{equation}\label{expectations:mixed}\begin{aligned}
\E((X_{n+1}-m_1)^4)&=M_4,&\E((X_{n+1}^2-m_2)^2)&=M_{2,2},\\
\E((X_{n+1}^3-m_3)(X_{n+1}-m_1))&=M_{1,3},&\E((X_{n+1}^2-m_2)(X_{n+1}-m_1)^2)&=M_{1,1,2}.
\end{aligned}\end{equation}
For the expectation of the recentered sums
\begin{equation}\label{expectations:tildesums}
\E(\wt S_n)=\E(\wt T_n)=\E(\wt U_n)=0
\end{equation}
holds.

Then we are ready to verify the recursions \eqref{recursion2}--\eqref{recursion4}.
We rewrite the $(n+1)$st terms on the left-hand side by \eqref{recursionrewrite} in each of the recursions.
Then we use \eqref{condexp1}--\eqref{condexp11} and \eqref{expectations:tildesums} to show \eqref{recursion2},
recursion \eqref{recursion12} follow by \eqref{condexp1}, \eqref{condexp2}, \eqref{condexp12} and \eqref{expectations:tildesums}.
By \eqref{condexp1}--\eqref{condexp111} and \eqref{expectations:tildesums} we get \eqref{recursion3}
and by \eqref{condexp1}, \eqref{condexp3}, \eqref{expectations:mixed}--\eqref{expectations:tildesums} we have \eqref{recursion13}.
Recursion \eqref{recursion22} follows from \eqref{condexp2} and \eqref{expectations:mixed}.
For \eqref{recursion112}, the equalities \eqref{condexp1}--\eqref{condexp11}, \eqref{condexp2}--\eqref{condexp12},
\eqref{expectations:mixed}--\eqref{expectations:tildesums} are used.
Finally formulas \eqref{condexp1}--\eqref{condexp111}, \eqref{expectations:mixed}--\eqref{expectations:tildesums} together yield \eqref{recursion4}.
\end{proof}

For the proof of Proposition~\ref{prop:recsolutions} about the solutions of recursions in Proposition~\ref{prop:recursions} one uses the following two lemmas.
The first one provides the general solution of recursions which the moments of the elephant random walk satisfy,
the second one contains two useful identities about sums of gamma ratios.

\begin{lemma}\label{lemma:recursions}
Assume that the real sequence $(b_n)$ satisfies the recursion relation
\begin{equation}\label{bnrecursion}
b_{n+1}=\left(1+\frac\beta n\right)b_n+c_n
\end{equation}
for $n=1,2,\dots$ for some given $\beta>0$ and given sequence $c_n$ for $n=1,2,\dots$ and with a given initial value $b_1$.
Then
\begin{equation}\label{bnexplicit}
b_n=\frac{\Gamma(n+\beta)}{\Gamma(n)}\left(\frac{b_1}{\Gamma(1+\beta)}+\sum_{j=1}^{n-1}\frac{\Gamma(j+1)}{\Gamma(j+1+\beta)}c_j\right)
\end{equation}
holds for $n=1,2,\dots$.
In particular, if $c_n=c$ for all $n=1,2,\dots$ and also $b_1=c$, then
\begin{equation}\label{bnexplicitsimple}
b_n=\frac c{(\beta-1)\Gamma(\beta)}\frac{\Gamma(n+\beta)}{\Gamma(n)}-\frac c{\beta-1}n
\end{equation}
for all $n=1,2,\dots$.
\end{lemma}

\begin{lemma}\label{lemma:gammasum}
Let $a$ and $b$ be two arbitrary non-negative real numbers such that $b\neq a+1$.
Then for all $n=1,2,\dots$, the following identities hold
\begin{align}
\sum_{j=1}^n\frac{\Gamma(j+a)}{\Gamma(j+b)}&=\frac1{b-a-1}\left(\frac{\Gamma(a+1)}{\Gamma(b)}-\frac{\Gamma(n+a+1)}{\Gamma(n+b)}\right),\label{gammasum1}\\
\sum_{j=1}^n\frac{\Gamma(j+a)}{\Gamma(j+b)}j&=\frac1{(b-a-1)(b-a-2)}\left(\frac{\Gamma(a+1)}{\Gamma(b-1)}-\frac{\Gamma(n+a+1)}{\Gamma(n+b-1)}\right)\label{gammasum2}\\
&\qquad-\frac{n}{b-a-1}\frac{\Gamma(n+1+a)}{\Gamma(n+b)}.\notag
\end{align}
\end{lemma}

\begin{proof}[Proof of Proposition~\ref{prop:recsolutions}]
The solutions \eqref{recsolution2}, \eqref{recsolution12}, \eqref{recsolution13} and \eqref{recsolution22} immediately follow from \eqref{bnexplicitsimple} of Lemma~\ref{lemma:recursions} with $\beta=2\alpha$ and with $b$ being equal to $M_2$, $M_{1,2}$, $M_{1,3}$ and $M_{2,2}$ respectively.

To get \eqref{recsolution3}, we apply Lemma~\ref{lemma:recursions} to \eqref{recursion3} with $\beta=3\alpha$, $b_1=M_3$ and
\begin{equation}\label{bn3}\begin{aligned}
c_n&=\frac{3\alpha}n\E(\wt S_n\wt T_n)-\frac{6\alpha}nm_1\E(\wt S_n^2)+M_3\\
&=\frac{3\alpha M_{1,2}-6\alpha m_1M_2}{(2\alpha-1)\Gamma(2\alpha)}\frac{\Gamma(n+2\alpha)}{\Gamma(n+1)}
-\frac{3\alpha M_{1,2}}{2\alpha-1}+\frac{6\alpha m_1M_2}{2\alpha-1}+M_3\\
&=\frac{3\alpha M_3}{(2\alpha-1)\Gamma(2\alpha)}\frac{\Gamma(n+2\alpha)}{\Gamma(n+1)}-\frac{(\alpha+1)M_3}{2\alpha-1}
\end{aligned}\end{equation}
where we used the solutions \eqref{recsolution2} and \eqref{recsolution12} in the second equality above and the identity $M_{1,2}-2m_1M_2=M_3$ in the last one.
With this value of $c_n$, the summation on the right-hand side of \eqref{bnexplicit} is
\begin{equation}\begin{aligned}
\sum_{j=1}^{n-1}\frac{\Gamma(j+1)}{\Gamma(j+1+\beta)}c_j
&=\frac{3\alpha M_3}{(2\alpha-1)\Gamma(2\alpha)}\sum_{j=1}^{n-1}\frac{\Gamma(j+2\alpha)}{\Gamma(j+1+3\alpha)}
-\frac{(\alpha+1)M_3}{2\alpha-1}\sum_{k=1}^{n-1}\frac{\Gamma(j+1)}{\Gamma(j+1+3\alpha)}\\
&=\frac{3M_3}{(2\alpha-1)\Gamma(2\alpha)}\left(\frac{\Gamma(2\alpha+1)}{\Gamma(3\alpha+1)}-\frac{\Gamma(n+2\alpha)}{\Gamma(n+3\alpha)}\right)\\
&\qquad-\frac{(\alpha+1)M_3}{(2\alpha-1)(3\alpha-1)}\left(\frac1{\Gamma(3\alpha+1)}-\frac{\Gamma(n+1)}{\Gamma(n+3\alpha)}\right)
\end{aligned}\end{equation}
with the use of \eqref{gammasum1} from Lemma~\ref{lemma:gammasum} in the last equality.
Substituting it to the right-hand side of \eqref{bnexplicit} one arrives at \eqref{recsolution3} after the simplification of the leading term.

The proof of \eqref{recsolution112} is similar.
We have $\beta=3\alpha$, $b_1=M_{1,1,2}$ and
\begin{equation}\label{bn112}\begin{aligned}
c_n&=\frac{2\alpha}n\E(\wt S_n\wt U_n)+\frac\alpha n\E(\wt T_n^2)-\frac{4\alpha}nm_1\E(\wt S_n\wt T_n)-\frac{2\alpha}nm_2\E(\wt S_n^2)+M_{1,1,2}\\
&=\frac{\alpha(2M_{1,3}+M_{2,2}-4m_1M_{1,2}-2m_2M_2)}{(2\alpha-1)\Gamma(2\alpha)}\frac{\Gamma(n+2\alpha)}{\Gamma(n+1)}\\
&\qquad+\frac\alpha{2\alpha-1}(-2M_{1,3}-M_{2,2}+4m_1M_{1,2}+2m_2M_2)+M_{1,1,2}\\
&=\frac{3\alpha M_{1,1,2}}{(2\alpha-1)\Gamma(2\alpha)}\frac{\Gamma(n+2\alpha)}{\Gamma(n+1)}-\frac{(\alpha+1)M_{1,1,2}}{2\alpha-1}
\end{aligned}\end{equation}
by using \eqref{recsolution13}, \eqref{recsolution22}, \eqref{recsolution12} and \eqref{recsolution2} in the second equality
and the identity $2M_{1,3}+M_{2,2}-4m_1M_{1,2}-2m_2M_2=3M_{1,1,2}$ in the last step above.
Then by comparing \eqref{bn3} with \eqref{bn112} we see that the recursion for $\E(\wt S_n^2\wt T_n)$ is formally the same as that for $\E(\wt S_n^3)$
with $M_3$ replaced by $M_{1,1,2}$ which proves \eqref{recsolution112}.

Finally we show \eqref{recsolution4}.
We use Lemma~\ref{lemma:recursions} with $\beta=4\alpha$, $b_1=M_4$ and
\begin{equation}\begin{aligned}
c_n&=\frac{6\alpha}n\E(\wt S_n^2\wt T_n)+\frac{4\alpha}n\E(\wt S_n\wt U_n)-\frac{12\alpha}nm_1\left(\E(\wt S_n^3)+\E(\wt S_n\wt T_n)\right)\\
&\qquad+\left(\frac{12\alpha}nm_1^2+6M_2\right)\E(\wt S_n^2)+M_4\\
&=\frac{\alpha(24M_{1,1,2}-48m_1M_3)}{(3\alpha-1)\Gamma(3\alpha)}\frac{\Gamma(n+3\alpha)}{\Gamma(n+1)}
+\frac{6M_2^2}{(2\alpha-1)\Gamma(2\alpha)}\frac{\Gamma(n+2\alpha)}{\Gamma(n)}\\
&\qquad+\frac{\alpha(-18M_{1,1,2}+4M_{1,3}+36m_1M_3-12m_1M_{1,2}+12m_1^2M_2)}{(2\alpha-1)\Gamma(2\alpha)}\frac{\Gamma(n+2\alpha)}{\Gamma(n+1)}
-\frac{6M_2^2}{2\alpha-1}n\\
&\qquad+\frac{\alpha(\alpha+1)(6M_{1,1,2}-12m_1M_3)}{(2\alpha-1)(3\alpha-1)}+\frac{\alpha(-4M_{1,3}+12m_1M_{1,2}-12m_1^2M_2)}{2\alpha-1}+M_4\\
&=\frac{24\alpha(M_4-M_2^2)}{(3\alpha-1)\Gamma(3\alpha)}\frac{\Gamma(n+3\alpha)}{\Gamma(n+1)}
+\frac{6M_2^2}{(2\alpha-1)\Gamma(2\alpha)}\frac{\Gamma(n+2\alpha)}{\Gamma(n)}\\
&\qquad+\frac{\alpha(18M_2^2-14M_4)}{(2\alpha-1)\Gamma(2\alpha)}\frac{\Gamma(n+2\alpha)}{\Gamma(n+1)}-\frac{6M_2^2}{2\alpha-1}n
+\frac{(5\alpha+1)M_4-6\alpha(\alpha+1)M_2^2}{(2\alpha-1)(3\alpha-1)}
\end{aligned}\end{equation}
where we used the solutions \eqref{recsolution112}, \eqref{recsolution13}, \eqref{recsolution3}, \eqref{recsolution12} and \eqref{recsolution2}
in the second equality and the identities $M_{1,1,2}-2m_1M_3=M_4-M_2^2$ and $M_{1,3}-3m_1M_{1,2}+3m_1^2M_2=M_4$ in the last equality above.
Then the summation on the right-hand side of equation \eqref{bnexplicit} can be given as follows
\begin{equation}\label{reccompute4}\begin{aligned}
&\sum_{j=1}^{n-1}\frac{\Gamma(j+1)}{\Gamma(j+1+\beta)}c_j\\
&\qquad=\frac{24\alpha(M_4-M_2^2)}{(3\alpha-1)\Gamma(3\alpha)}\sum_{j=1}^{n-1}\frac{\Gamma(j+3\alpha)}{\Gamma(j+1+4\alpha)}
+\frac{6M_2^2}{(2\alpha-1)\Gamma(2\alpha)}\sum_{j=1}^{n-1}j\frac{\Gamma(j+2\alpha)}{\Gamma(j+1+4\alpha)}\\
&\qquad\qquad+\frac{\alpha(18M_2^2-14M_4)}{(2\alpha-1)\Gamma(2\alpha)}\sum_{j=1}^{n-1}\frac{\Gamma(j+2\alpha)}{\Gamma(j+1+4\alpha)}
-\frac{6M_2^2}{2\alpha-1}\sum_{j=1}^{n-1}j\frac{\Gamma(j+1)}{\Gamma(j+1+4\alpha)}\\
&\qquad\qquad+\frac{(5\alpha+1)M_4-6\alpha(\alpha+1)M_2^2}{(2\alpha-1)(3\alpha-1)}\sum_{j=1}^{n-1}\frac{\Gamma(j+1)}{\Gamma(j+1+4\alpha)}\\
&\qquad\sim\frac{24\alpha(M_4-M_2^2)}{(3\alpha-1)\Gamma(3\alpha)}\frac1\alpha\frac{\Gamma(3\alpha+1)}{\Gamma(4\alpha+1)}
+\frac{6M_2^2}{(2\alpha-1)\Gamma(2\alpha)}\frac1{2\alpha(2\alpha-1)}\frac{\Gamma(2\alpha+1)}{\Gamma(4\alpha)}\\
&\qquad\qquad+\frac{\alpha(18M_2^2-14M_4)}{(2\alpha-1)\Gamma(2\alpha)}\frac1{2\alpha}\frac{\Gamma(2\alpha+1)}{\Gamma(4\alpha+1)}
-\frac{6M_2^2}{2\alpha-1}\frac1{(4\alpha-1)(4\alpha-2)\Gamma(4\alpha)}\\
&\qquad\qquad+\frac{(5\alpha+1)M_4-6\alpha(\alpha+1)M_2^2}{(2\alpha-1)(3\alpha-1)}\frac1{(4\alpha-1)\Gamma(4\alpha+1)}\\
\end{aligned}\end{equation}
where the asymptotic equality above follows since
\begin{equation}\label{sumasympt}
\sum_{j=1}^{n-1}\frac{\Gamma(j+3\alpha)}{\Gamma(j+1+4\alpha)}\sim\frac1\alpha\frac{\Gamma(3\alpha+1)}{\Gamma(4\alpha+1)},\quad
\sum_{j=1}^{n-1}j\frac{\Gamma(j+2\alpha)}{\Gamma(j+1+4\alpha)}\sim\frac1{2\alpha(2\alpha-1)}\frac{\Gamma(2\alpha+1)}{\Gamma(4\alpha)}
\end{equation}
holds as $n\to\infty$ and from three other similar asymptotic equalities corresponding to the summations in further terms of \eqref{reccompute4}.
These asymptotics can be seen from Lemma~\ref{lemma:gammasum} by neglecting the terms vanishing in the $n\to\infty$ limit.
By substituting \eqref{reccompute4} into \eqref{bnexplicit} we see that as $n\to\infty$,
$\E(\wt S_n^4)$ is asymptotically equal to a constant times $\Gamma(n+4\alpha)/\Gamma(n)\sim n^{4\alpha}$.
The value of the constant is obtained by adding $b_1/\Gamma(4\alpha+1)=M_4/\Gamma(4\alpha+1)$ to the expression in \eqref{reccompute4}.
This verifies that vanishing terms in \eqref{reccompute4} can be disregarded.
Straghtforward simplification of the sum of $M_4/\Gamma(4\alpha+1)$ and the right-hand side of \eqref{reccompute4} yields the coefficient of $\Gamma(n+4\alpha)/\Gamma(n)$
on the right-hand side of \eqref{recsolution4} which completes the proof.
\end{proof}

\begin{proof}[Proof of Lemma~\ref{lemma:recursions}]
The formula \eqref{bnexplicit} with $n$ replaced by $n+1$ can be written as
\begin{equation}
b_{n+1}=\frac{(n+\beta)\Gamma(n+\beta)}{n\Gamma(n)}\left(\frac{b_1}{\Gamma(1+\beta)}+\sum_{j=1}^n\frac{\Gamma(j+1)}{\Gamma(j+1+\beta)}c_j\right)
=\left(1+\frac\beta n\right)b_n+c_n
\end{equation}
where the last equality follows by separating the $j=n$ term in the sum.
Since the $n=1$ case of \eqref{bnexplicit} gives back $b_1$, the first part of the lemma is proved.

In the special case of constant $c_n$, the right-hand side of \eqref{bnexplicit}
by using \eqref{gammasum1} in Lemma~\ref{lemma:gammasum} with $a=1$ and $b=1+\beta$ can be written as
\begin{equation}\begin{aligned}
b_n&=\frac{\Gamma(n+\beta)}{\Gamma(n)}
\left(\frac c{\Gamma(1+\beta)}+\frac c{\beta-1}\left(\frac1{\Gamma(1+\beta)}-\frac{\Gamma(n+1)}{\Gamma(n+\beta)}\right)\right)\\
&=\frac{\Gamma(n+\beta)}{\Gamma(n)}\frac{c\beta}{(\beta-1)\Gamma(1+\beta)}-\frac c{\beta-1}n
\end{aligned}\end{equation}
which reduces to \eqref{bnexplicitsimple} and completes the proof.
\end{proof}

\begin{proof}[Proof of Lemma~\ref{lemma:gammasum}]
The difference of the right-hand side of \eqref{gammasum1} for $n$ and for $n-1$ instead of $n$ is
\begin{equation}\begin{aligned}
\frac1{b-a-1}\left(\frac{\Gamma(n+a)}{\Gamma(n+b-1)}-\frac{\Gamma(n+a+1)}{\Gamma(n+b)}\right)
&=\frac1{b-a-1}\frac{\Gamma(n+a)}{\Gamma(n+b-1)}\left(1-\frac{n+a}{n+b-1}\right)\\
&=\frac{\Gamma(n+a)}{\Gamma(n+b)}
\end{aligned}\end{equation}
which proves \eqref{gammasum1} since both sides are $0$ for $n=0$.
To see \eqref{gammasum2} we can write $j=\sum_{k=1}^j1$ on the left-hand side.
Exchanging the order of summations gives that
\begin{equation}\begin{aligned}
\sum_{j=1}^n\frac{\Gamma(j+a)}{\Gamma(j+b)}j&=\sum_{k=1}^n\sum_{j=k}^n\frac{\Gamma(j+a)}{\Gamma(j+b)}\\
&=\sum_{k=1}^n\frac1{b-a-1}\left(\frac{\Gamma(k+a)}{\Gamma(k+b-1)}-\frac{\Gamma(n+a+1)}{\Gamma(n+b)}\right)\\
&=\frac1{b-a-1}\left(\frac1{b-a-2}\left(\frac{\Gamma(a+1)}{\Gamma(b-1)}-\frac{\Gamma(n+a+1)}{\Gamma(n+b-1)}\right)-n\frac{\Gamma(n+a+1)}{\Gamma(n+b)}\right)
\end{aligned}\end{equation}
where we used \eqref{gammasum1} in the last two steps repeatedly.
\end{proof}

\bibliography{elephant}
\bibliographystyle{alpha}

\end{document}